\documentclass[11pt,leqno]{article}
\usepackage{hyperref}
\usepackage{soul}
\usepackage[doc]{optional}
\usepackage{xcolor}
\definecolor{labelkey}{rgb}{0,0.08,0.45}
\definecolor{rekey}{rgb}{0,0.6,0.0}
\definecolor{Brown}{rgb}{0.45,0.0,0.05}
\usepackage{exscale,relsize}
\usepackage{amsmath}
\usepackage{amsfonts}
\usepackage{amssymb}
\usepackage{calc}
\usepackage{theorem}
\usepackage{pifont}      
\usepackage{graphicx}
\DeclareMathOperator{\weakstarly}{\stackrel{\mathrm{w*}}{\rightharpoondown}}

\newcommand{\wk}{\ensuremath{\operatorname{w*}}}

\newcommand{\infconv}{\ensuremath{\mbox{\small$\,\square\,$}}}

\newcommand{\scal}[2]{\langle{{#1},{#2}}\rangle}

\newcommand{\RR}{\ensuremath{\mathbb R}}

\newcommand{\RX}{\ensuremath{\,\left]-\infty,+\infty\right]}}
\newcommand{\RXX}{\ensuremath{\,\left[-\infty,+\infty\right]}}

\newcommand{\NN}{\ensuremath{\mathbb N}}

\newcommand{\thalb}{\ensuremath{\tfrac{1}{2}}}

\newcommand{\menge}[2]{\big\{{#1} \mid {#2}\big\}}

\newcommand{\To}{\ensuremath{\rightrightarrows}}

\newcommand{\spand}{\operatorname{span}}

\newcommand{\dom}{\ensuremath{\operatorname{dom}}}

\newcommand{\gra}{\ensuremath{\operatorname{gra}}}

\newcommand{\inte}{\ensuremath{\operatorname{int}}}

\newcommand{\ran}{\ensuremath{\operatorname{ran}}}

\renewcommand{\phi}{\ensuremath{\varphi}}

\newtheorem{theorem}{Theorem}[section]

\newtheorem{fact}[theorem]{Fact}
\newtheorem{corollary}[theorem]{Corollary}
\newtheorem{proposition}[theorem]{Proposition}
\newtheorem{definition}[theorem]{Definition}

\theoremstyle{plain}{\theorembodyfont{\rmfamily}
}
\theoremstyle{plain}{\theorembodyfont{\rmfamily}
}
\theoremstyle{plain}{\theorembodyfont{\rmfamily}
}
\theoremstyle{plain}{\theorembodyfont{\rmfamily}
\newtheorem{example}[theorem]{Example}}
\theoremstyle{plain}{\theorembodyfont{\rmfamily}
\newtheorem{remark}[theorem]{Remark}}

\theoremstyle{plain}{\theorembodyfont{\rmfamily}
}

\def\endproof{\ensuremath{\quad \hfill \blacksquare}}

\begin{document}


\title{\sffamily{The Br\'ezis-Browder Theorem in a general Banach space}}

\author{
Heinz H.\ Bauschke\thanks{Mathematics, Irving K.\ Barber School,
University of British Columbia, Kelowna, B.C. V1V 1V7, Canada. E-mail:
\texttt{heinz.bauschke@ubc.ca}.},\;
Jonathan M. Borwein\thanks{CARMA, University of Newcastle,
 Newcastle, New South Wales 2308, Australia. E-mail:
\texttt{jonathan.borwein@newcastle.edu.au}. Distinguished Professor
King Abdulaziz University, Jeddah.},\;
 Xianfu
Wang\thanks{Mathematics, Irving K.\ Barber School, University of British Columbia,
Kelowna, B.C. V1V 1V7, Canada. E-mail:
\texttt{shawn.wang@ubc.ca}.},\; and Liangjin\
Yao\thanks{Mathematics, Irving K.\ Barber School, University of British Columbia,
Kelowna, B.C. V1V 1V7, Canada.
E-mail:  \texttt{ljinyao@interchange.ubc.ca}.}}

\date{October 25,  2011}
\maketitle

\begin{abstract} \noindent
During the 1970s Br\'ezis and Browder presented a now classical characterization
of maximal monotonicity of monotone linear relations in reflexive spaces.
In this paper, we  extend and refine their result to a general Banach space.
\end{abstract}

\noindent {\bfseries 2010 Mathematics Subject Classification:}\\
{Primary  47A06, 47H05;
Secondary 47B65, 47N10, 90C25}

\noindent {\bfseries Keywords:}
Adjoint,
Br\'ezis-Browder Theorem,
Fenchel conjugate,
linear relation,
maximally monotone operator,
monotone operator,
  operator of type (D),
 operator of type (FP),
operator of type (NI),
set-valued operator,
 skew operator, symmetric operator.

\section{Introduction}

Throughout this paper, we assume that
\begin{align*}\text{$X$ is a real Banach space with norm $\|\cdot\|$,}
\end{align*}
that $X^*$ is the continuous dual of $X$,
 and
that $X$ and $X^*$ are paired by $\scal{\cdot}{\cdot}$.
The \emph{closed unit ball} in $X$ is denoted by $B_X=
\menge{x\in X}{\|x\|\leq1}$, and $\NN=\{1,2,3,\ldots\}$.

We identify
$X$ with its canonical image in the bidual space $X^{**}$.
As always, $X\times X^*$ and $(X\times X^*)^* = X^*\times X^{**}$
are paired via $$\scal{(x,x^*)}{(y^*,y^{**})} =
\scal{x}{y^*} + \scal{x^*}{y^{**}},$$ where $(x,x^*)\in X\times X^*$
and $(y^*,y^{**}) \in X^*\times X^{**}$.

Let $A\colon X\To X^*$
be a \emph{set-valued operator} (also known as multifunction)
from $X$ to $X^*$, i.e., for every $x\in X$, $Ax\subseteq X^*$,
and let
$\gra A = \menge{(x,x^*)\in X\times X^*}{x^*\in Ax}$ be
the \emph{graph} of $A$. The \emph{domain} of $A$, written as
 $\dom A$,  is $\dom A= \menge{x\in X}{Ax\neq\varnothing}$ and
$\ran A=A(X)$ is the \emph{range} of $A$. We say $A$ is a
 \emph{linear relation} if $\gra A$ is a linear
subspace.
Now let $U\times V\subseteq X\times X^*$.
We say that $A$ is \emph{monotone} with respect to $U\times V$,
if for every $(x,x^*)\in(\gra A)\cap (U\times V)$
and $(y,y^*)\in(\gra A)\cap (U\times V)$, we have
\begin{equation}
\scal{x-y}{x^*-y^*}\geq 0.
\end{equation}
Of course, by (classical) monotonicity we mean
monotonicity with respect to $X\times X^*$.
Furthermore, we say that
$A$ is \emph{maximally monotone} with respect to $U\times V$
if $A$ is monotone with respect to $U\times V$ and for
every operator $B\colon X\To X^*$ that is monotone with respect to $U\times
V$ and such that $(\gra A)\cap (U\times V)\subseteq (\gra B)\cap (U\times
V)$, we necessarily have $(\gra A)\cap (U\times V)=(\gra B)\cap(U\times
V)$. Thus, (classical) maximal monotonicity corresponds to
maximal monotonicity with respect to $X\times X^*$.
This slightly unusual presentation is required to state our main results;
moreover,
it yields a more concise formulation of monotone operators of type
(FP).

Now let $A:X\rightrightarrows X^*$ be monotone
and $(x,x^*)\in X\times X^*$.
 We say $(x,x^*)$ is \emph{monotonically related to}
$\gra A$ if
\begin{align*}
\langle x-y,x^*-y^*\rangle\geq0,\quad \forall (y,y^*)\in\gra A.\end{align*}
If $Z$ is
a real  Banach space with continuous dual $Z^*$ and a subset $S$ of
$Z$, we denote $S^\bot$ by $S^\bot = \menge{z^*\in Z^*}{\langle
z^*,s\rangle= 0,\quad \forall s\in S}$. Given a subset $D$ of  $Z^*$,
we define $D_{\bot}$  by
$D_\bot = \menge{z\in Z}{\langle z,d^*\rangle= 0,
\quad \forall d^*\in D}=D^{\bot}\cap Z$.

The operator \emph{adjoint} of $A$,
written as $A^*$, is defined by
\begin{equation*}
\gra A^* =
\menge{(x^{**},x^*)\in X^{**}\times X^*}{(x^*,-x^{**})\in(\gra A)^{\bot}}.
\end{equation*}
Note that the adjoint is always a linear relation
with $\gra A^*\subseteq X^{**}\times X^* \subseteq X^{**}\times X^{***}$.
These inclusions make it possible to consider
monotonicity properties of $A^*$; however, care is required:
as a linear relation, $\gra A^* \subseteq X^{**}\times X^*$ while
as a potential monotone operator we are led to consider
$\gra A^* \subseteq X^{**}\times X^{***}$.
Now let $A:X\rightrightarrows X^*$ be a linear relation.
 We say that $A$ is
\emph{skew} if $\gra A \subseteq \gra (-A^*)$;
equivalently, if $\langle x,x^*\rangle=0,\; \forall (x,x^*)\in\gra A$.
Furthermore,
$A$ is \emph{symmetric} if $\gra A
\subseteq\gra A^*$; equivalently, if $\scal{x}{y^*}=\scal{y}{x^*}$,
$\forall (x,x^*),(y,y^*)\in\gra A$.

We now recall three fundamental subclasses of maximally
monotone operators.

 \begin{definition}
 Let $A:X\To X^*$ be maximally monotone.
 Then three key types of monotone operators are defined as follows.
 \begin{enumerate}
 \item $A$ is
\emph{of dense type or type (D)} (1971, \cite{Gossez3}) if for every
$(x^{**},x^*)\in X^{**}\times X^*$ with
\begin{align*}
\inf_{(a,a^*)\in\gra A}\langle a-x^{**}, a^*-x^*\rangle\geq 0,
\end{align*}
there exist a  bounded net
$(a_{\alpha}, a^*_{\alpha})_{\alpha\in\Gamma}$ in $\gra A$
such that
$(a_{\alpha}, a^*_{\alpha})_{\alpha\in\Gamma}$
weak*$\times$strong converges to
$(x^{**},x^*)$.
\item $A$ is
\emph{of type negative infimum (NI)} (1996, \cite{SiNI}) if
\begin{align*}
\sup_{(a,a^*)\in\gra A}\big(\langle a,x^*\rangle+\langle a^*,x^{**}\rangle
-\langle a,a^*\rangle\big)
\geq\langle x^{**},x^*\rangle,
\quad \forall(x^{**},x^*)\in X^{**}\times X^*.
\end{align*}
\item
$A$ is \emph{of type Fitzpatrick-Phelps (FP)} (1992, \cite{FP}) if
whenever $V$ is an open convex subset of $X^*$ such that $V\cap \ran
A\neq\varnothing$,
it must follow
that $A$ is maximally monotone with respect to $X\times V$.
\end{enumerate}

\end{definition}
\begin{fact}\emph{(See \cite{Si, Si2, BorVan}.)}\label{ITPR}
The following are maximally monotone of type (D), (NI), and (FP).
\begin{enumerate}
\item $\partial f$, where
 $f: X\to\RX$ is convex, lower semicontinuous, and proper;
\item $A\colon X\To X^*$, where $A$ is maximally monotone and $X$ is
reflexive.
\end{enumerate}
\end{fact}
These and other relationships known
amongst these and other monotonicity notions are described in
\cite[Chapter 9]{BorVan}.
As we see in \cite{BBWY2} and \cite{Si4, SiNI, MarSva}, it is now known that the
three classes coincide.

Monotone operators have proven to be a key class of objects in both
modern Optimization and Analysis; see, e.g., \cite{Bor1,Bor2,Bor3},
the books \cite{BC2011,
BorVan,BurIus,ph,Si,Si2,RockWets,Zalinescu,Zeidlerlin,Zeidler}
and the references therein.

Let us now precisely describe the aforementioned Br\'{e}zis-Browder Theorem:

\begin{theorem}[Br\'ezis-Browder in reflexive Banach space \cite{BB76,Brezis-Browder}]
 \label{thm:brbr} Suppose that $X$ is reflexive.
Let $A\colon X \To X^*$ be a monotone linear relation
such that $\gra A$ is closed.
Then
$A$ is maximally monotone if and only if the adjoint
$A^*$ is monotone.
\end{theorem}

In this paper, we generalize the Br\'{e}zis-Browder Theorem to an arbitrary Banach space.
(See \cite{Si3} for Simons' recent extension of the above result
to {symmetrically self-dual Banach spaces} (SSDB)
spaces as defined in \cite[\S21]{Si2}.)

Our main result is the following.

\begin{theorem}[Br\'ezis-Browder in general Banach space] \label{thm:bbwy}
Let $A\colon X\rightrightarrows X^*$ be a monotone linear
 relation such that $\gra A$ is closed.
Then the following are equivalent.
\begin{enumerate}
\item $A$ is maximally monotone of type (D).
\item  $A$ is  maximally monotone  of type (NI).
\item  $A$ is  maximally monotone  of type (FP).
\item  $A^*$ is monotone.
\end{enumerate}
\end{theorem}

In Section~\ref{s:aux}, we collect auxiliary results for future
reference and for the reader's convenience.
In Section~\ref{s:tech}, we provide the key technical step showing that when $A^*$ is monotone then $A$ is of type (D).
Our central result, the
generalized
Br\'ezis-Browder Theorem
(Theorem~\ref{thm:bbwy}), is then proved in Section~\ref{s:main}.
Finally, in Section \ref{s:skew} with the necessary proviso that the domain be closed, we establish further results
such as Theorem~\ref{BrBrTDS:1} relating to the skew part of the operator. This was motivated by and
 extends \cite[Theorem~4.1]{BB} which studied the case of a bounded linear operator.

Finally, let us mention that we adopt standard convex analysis notation.
Given a subset $C$ of $X$,
$\inte C$ is the \emph{interior} of $C$,
$\overline{C}$ is   the
\emph{norm closure} of $C$.
 For the set $D\subseteq X^*$, $\overline{D}^{\wk}$
is the weak$^{*}$ closure of $D$.
If $E\subseteq X^{**}$, $\overline{E}^{\wk}$
  is the weak$^{*}$ closure of $E$ in $X^{**}$ with
 the topology induced by $X^*$.
The \emph{indicator function} of $C$, written as $\iota_C$, is defined
at $x\in X$ by
\begin{align}
\iota_C (x)=\begin{cases}0,\,&\text{if $x\in C$;}\\
+\infty,\,&\text{otherwise}.\end{cases}\end{align} For every $x\in
X$, the \emph{normal cone} operator of $C$ at $x$ is defined by
$N_C(x)= \menge{x^*\in X^*}{\sup_{c\in C}\scal{c-x}{x^*}\leq 0}$, if
$x\in C$; and $N_C(x)=\varnothing$, if $x\notin C$.

Let $f\colon X\to \RX$. Then
$\dom f= f^{-1}(\RR)$ is the \emph{domain} of $f$, and
$f^*\colon X^*\to\RXX\colon x^*\mapsto
\sup_{x\in X}(\scal{x}{x^*}-f(x))$ is
the \emph{Fenchel conjugate} of $f$.
The  \emph{lower semicontinuous hull} of $f$ is denoted by $\overline{f}$.
We say $f$ is proper if $\dom f\neq\varnothing$.
Let $f$ be proper. The \emph{subdifferential} of
$f$ is defined by
   $$\partial f\colon X\To X^*\colon
   x\mapsto \{x^*\in X^*\mid(\forall y\in
X)\; \scal{y-x}{x^*} + f(x)\leq f(y)\}.$$
For $\varepsilon \geq 0$,
the \emph{$\varepsilon$--subdifferential} of $f$ is defined by
   $\partial_{\varepsilon} f\colon X\To X^*\colon
   x\mapsto \menge{x^*\in X^*}{(\forall y\in
X)\; \scal{y-x}{x^*} + f(x)\leq f(y)+\varepsilon}$.
Note that $\partial f = \partial_{0}f$.
We  denote  by $J:=J_X$ the duality map, i.e.,
the subdifferential of the function $\tfrac{1}{2}\|\cdot\|^2$
mapping $X$ to $X^*$. For the properties of $J$, see \cite[Example~2.26]{ph}.

Let $(z,z^*)\in X\times X^*$ and $F\colon X\times X^*\to\RX$.
Then $F_{(z,z^*)}:X\times X^*\rightarrow\RX$
 \cite{MLegSva,Si2} is defined by
\begin{align}
F_{(z,z^*)}(x,x^*)&=F(z+x,z^*+x^*)-\big(\langle x,z^*\rangle+
\langle z,x^*\rangle+\langle z,z^*\rangle\big)\notag\\
&=F(z+x,z^*+x^*)-\langle z+x,z^*+x^*\rangle+\langle x,x^*\rangle,
\quad \forall(x,x^*)\in X\times X^*.\label{e:timmy}
\end{align}
Let now $Y$ be another real Banach space. We set  $P_X: X\times Y\rightarrow
X\colon (x,y)\mapsto x$.
Let $F_1, F_2\colon X\times Y\rightarrow\RX$.
Then the \emph{partial inf-convolution}  $F_1\Box_2 F_2$
is the function defined on $X\times Y$ by
\begin{equation*}F_1\Box_2 F_2\colon
(x,y)\mapsto \inf_{v\in Y}
F_1(x,y-v)+F_2(x,v).
\end{equation*}

 \section{Prerequisite results}\label{s:aux}

\begin{fact}\emph{(See \cite[Proposition~2.6.6(c)]{Megg} or
 \cite[Theorem~4.7 and Theorem~3.12]{Rudin}.)}\label{rudin:1}
Let $C$ be a subspace of  $X$, and $D$ be a subspace of $X^*$.
Then \begin{align*}(C^{\bot})_{\bot}=\overline{C}
\quad\text{and}\quad(D_{\bot})^{\bot}=\overline{D}^{\wk}.
\end{align*}
\end{fact}

 \begin{fact}[Rockafellar] \label{f:F4}
\emph{(See \cite[Theorem~3(b)]{Rock66}, \cite[Theorem~18.1]{Si2}
or
\cite[Theorem~2.8.7(iii)]{Zalinescu}.)}
Let $f,g:  X\rightarrow\RX$ be proper convex functions.
Assume that there exists a point $x_0\in\dom f \cap \dom g$
such that $g$ is continuous at $x_0$. Then
$$\partial (f+g)=\partial f+\partial g.$$
\end{fact}

 \begin{fact}[Br\o ndsted-Rockafellar]\label{f:FJB5}
\emph{(See  \cite[Theorem~3.1.2 or Theorem~3.1.4(ii)]{Zalinescu}.)}
Let $f: X\rightarrow\RX$ be a proper lower semicontinuous and convex function
and $x^*\in\dom f^*$.
Then
there exists   a
sequence $(x_n,x^*_n)_{n\in\NN}$ in $\gra \partial f$ such that
$x^*_n \to x^*$.
\end{fact}

\begin{fact}[Attouch-Br\'ezis]\emph{(See \cite[Theorem~1.1]{AtBrezis}
 or \cite[Remark~15.2]{Si2}.)}\label{AttBre:1}
Let $f,g: X\rightarrow\RX$ be proper lower semicontinuous  and convex.
Assume that $
\bigcup_{\lambda>0} \lambda\left[\dom f-\dom g\right]$
is a closed subspace of $X$.
Then
\begin{equation*}
(f+g)^*(z^*) =\min_{y^*\in X^*} \left[f^*(y^*)+g^*(z^*-y^*)\right],\quad \forall z^*\in X^*.
\end{equation*}
\end{fact}

\begin{fact}[Simons and Z\u{a}linescu]
\emph{(See \cite[Theorem~4.2]{SiZ} or \cite[Theorem~16.4(a)]{Si2}.)}\label{F4}
Let $Y$ be a real Banach space and $F_1, F_2\colon X\times Y \to \RX$ be proper,
lower semicontinuous, and convex. Assume that
for every $(x,y)\in X\times Y$,
\begin{equation*}(F_1\Box_2 F_2)(x,y)>-\infty
\end{equation*}
and that  $\bigcup_{\lambda>0} \lambda\left[P_X\dom F_1-P_X\dom F_2\right]$
is a closed subspace of $X$. Then for every $(x^*,y^*)\in X^*\times Y^*$,
\begin{equation*}
(F_1\Box_2 F_2)^*(x^*,y^*)=\min_{u^*\in X^*}
\left[F_1^*(x^*-u^*,y^*)+F_2^*(u^*,y^*)\right].
\end{equation*}\end{fact}

The following result was first established in
\cite[Theorem~7.4]{Borwein}.
Now we give a new proof.
\begin{fact}[Borwein]\label{LinAdSum}
Let $A, B:X\rightrightarrows X^*$ be linear relations such that $\gra A$ and $\gra B$ are closed.
Assume that
$\dom A-\dom B$ is closed. Then
\begin{align*}(A+B)^*=A^*+B^*.\end{align*}
\end{fact}
\begin{proof}
We have
\begin{align}
\iota_{\gra (A+B)}=\iota_{\gra A}\Box_2\iota_{\gra B}.\label{LinAdSumP:1}
\end{align}
Let $(x^{**},x^*)\in X^{**}\times X^*$.
Since
$\gra A$ and $\gra B$ are closed convex,
$\iota_{\gra A}$ and  $\iota_{\gra B}$ are proper lower semicontinuous and convex.
Then by Fact~\ref{F4} and \eqref{LinAdSumP:1},
 there exists $y^*\in X^*$ such that
\begin{align}
\iota_{\gra(A+B)^*}(x^{**},x^*)&=\iota_{\big(\gra (A+B)\big)^{\bot}}(-x^*,x^{**})\nonumber\\
&=\iota^*_{\gra (A+B)}(-x^*,x^{**})\quad\text{(since $\gra (A+B)$ is a subspace)}\nonumber\\
&=
\iota^*_{\gra A}(y^*,x^{**})+\iota^*_{\gra B}(-x^*-y^*,x^{**})\nonumber\\
&=\iota_{(\gra A)^{\bot}}(y^*,x^{**})+\iota_{(\gra B)^{\bot}}(-x^*-y^*,x^{**})\nonumber\\
&=\iota_{\gra A^*}(x^{**},-y^*)+\iota_{\gra B^*}(x^{**},x^*+y^*)\nonumber\\
&=\iota_{\gra (A^*+B^*)}(x^{**},x^*).\label{LinAdSumP:2}
\end{align}
Then we have $\gra (A+B)^*=\gra (A^*+B^*)$ and hence $(A+B)^*=A^*+B^*$.
\end{proof}

\begin{fact}[Simons]
\emph{(See \cite[Lemma~19.7 and Section~22]{Si2}.)}
\label{f:referee}
Let $A:X\rightrightarrows X^*$ be a monotone operator such
 that $\operatorname{gra} A$ is convex with $\operatorname{gra} A
\neq\varnothing$.
Then the function
\begin{equation}
g\colon X\times X^* \rightarrow \left]-\infty,+\infty\right]\colon
(x,x^*)\mapsto \langle x, x^*\rangle + \iota_{\operatorname{gra} A}(x,x^*)
\end{equation}
is proper and convex.
\end{fact}

We also recall the somewhat more precise version of Theorem \ref{thm:brbr}.

\begin{fact}[Br\'ezis and Browder]  \label{Sv:7}\emph{(See \cite[Theorem~2]{Brezis-Browder},
 or \cite{Brezis70,BB76,Si3,Yao}.)}
Suppose that $X$ is reflexive.
Let $A\colon X \To X^*$ be a monotone linear relation
such that $\gra A$ is closed. Then the following are equivalent.
\begin{enumerate}
\item
$A$ is maximally monotone.
\item
$A^*$ is maximally monotone.
\item
$A^*$ is monotone.
\end{enumerate}
\end{fact}

This has a recent non-reflexive counterpart:

\begin{fact}\emph{(See \cite[Theorem~3.1]{BBWY:1}.)}
\label{TypeD:1}
Let $A:X\rightrightarrows X^*$ be a maximally monotone linear relation.
Then the following are equivalent.
\begin{enumerate}
\item $A$ is of type (D).

\item  $A$ is of type (NI).

\item  $A$ is of type (FP).
\item $A^*$ is monotone

\end{enumerate}
\end{fact}

Comparing of Fact~\ref{TypeD:1} and Fact~\ref{Sv:7},
we observe that the
 hypothesis in the latter (maximality of $A$) is more restrictive than in
  the former (closedness of the graph).
  In \cite[Theorem~3.1]{BBWY:1} we were unable to attack this issue.
The result of the next section redresses our lacuna.

Now let us cite some basic properties of linear relations.

The following result appeared in
 Cross' book \cite{Cross}.  We give new proofs of \ref{Sia:1}--\ref{Th:32}.
 The proof of the \ref{Th:32} below
was adapted from \cite[Remark~2.2]{BWY8}.

\begin{fact}\label{Rea:1}
Let $A:X \rightrightarrows X^*$ be a linear relation.
Then the following hold.
\begin{enumerate}
\item \label{Th:28}$Ax=x^* +A0,\quad\forall x^*\in Ax.$
\item \label{Th:30}
$A(\alpha x+\beta y)=\alpha Ax+\beta Ay, \forall (\alpha,\beta)\in\RR^2\smallsetminus\{(0,0)\},\;
\forall x,y\in\dom A$.
\item \label{Sia:2b}
$\langle A^*x,y\rangle=\langle x,Ay\rangle$ is a singleton, $\forall x\in \dom A^*, \forall y\in\dom A$.
\item \label{Sia:1}$(\dom A)^{\bot}=A^*0$ is (weak$^*$) closed
 and $\overline{\dom A}=(A^*0)_{\perp}$.
\item \label{Sia:2} If $\gra A$ is  closed, then $(\dom A^*)_{\bot}=A0$ and
$\overline{\dom A^*}^{\wk}=(A0)^{\perp}$.
\item\label{Th:32} If $\dom A$ is closed, then $\dom A^*=(\bar{A}0)^{\bot}$
 and thus $\dom A^*$ is (weak$^*$) closed, where $\bar{A}$ is the linear
relation whose graph is the closure of the graph of $A$.
\end{enumerate}
\end{fact}
\begin{proof}
\ref{Th:28}: See \cite[Proposition I.2.8(a)]{Cross}.
 \ref{Th:30}: See \cite[Corollary I.2.5]{Cross}.
\ref{Sia:2b}: See \cite[Proposition III.1.2]{Cross}.

\ref{Sia:1}: We have
\begin{align*}
x^*\in A^*0\Leftrightarrow(x^*,0)\in(\gra A)^{\bot}
\Leftrightarrow x^*\in(\dom A)^{\bot}.
\end{align*}
 Hence $(\dom A)^{\bot}= A^*0$ and thus $A^*0$ is weak$^*$ closed.
 By Fact~\ref{rudin:1},
$\overline{\dom A}=(A^*0)_{\bot}$.

\ref{Sia:2}: Using Fact~\ref{rudin:1},  \begin{align*}
&x^*\in A0\Leftrightarrow (0,x^*)\in\gra A=
 \left[(\gra A)^{\bot}\right]_{\bot}=\left[\gra -(A^*)^{-1}\right]_{\bot}
 \Leftrightarrow x^*\in(\dom A^*)_{\bot}.
\end{align*}
Hence $(\dom A^*)_{\bot}=A0$ and thus, by Fact~\ref{rudin:1},
 $\overline{\dom A^*}^{\wk}=(A0)^{\bot}$.

 \ref{Th:32}:
Let $\bar{A}$ be the linear
relation whose graph is the closure of the graph of $A$.
Then $\dom A = \dom\bar{A}$ and $A^* = \bar{A}^*$.
Then by Fact~\ref{AttBre:1},
\begin{equation*}
\iota_{X^*\times (\bar{A}0)^\perp} =
\iota_{\{0\}\times \bar{A}0}^* =
\big(\iota_{\gra \bar{A}} + \iota_{\{0\}\times X^*}\big)^*
= \iota_{\gra (-\bar{A}^*)^{-1}} \infconv \iota_{X^*\times\{0\}}
= \iota_{X^*\times \dom \bar{A}^*}.
\end{equation*}
It is clear that $\dom A^*=\dom \bar{A}^* = (\bar{A}0)^\perp$ is
weak$^*$ closed, hence closed.
\end{proof}

\section{A key result}\label{s:tech}

The proof of Proposition~\ref{PrTD:1} below  was partially inspired by
 that of \cite[Theorem~3.1]{BBWY:1}.

\begin{proposition}\label{PrTD:1}
Let $A:X\rightrightarrows X^*$ be a monotone linear relation such that $\gra A$ is closed
and $A^*$ is monotone. Then $A$ is maximally monotone of type (D).
\end{proposition}

\begin{proof}
By Fact~\ref{TypeD:1}, it suffices to show that $A$ is maximally monotone.
Let $(z,z^*)\in X\times X^*$. Assume that
\begin{align}\text{$(z,z^*)$ is monotonically related to
$\gra A$.}\label{TIN:1}
\end{align}
Define
\begin{align*}
F:X\times X^*\to \RX\colon
(x,x^*)\mapsto \iota_{\gra A}(x,x^*)+\langle x,x^*\rangle.
\end{align*}
Fact~\ref{f:referee} implies that
$F$ is convex and since $\gra A$ is closed, $F$ is also proper, lower semicontinuous.
Recalling \eqref{e:timmy}, note that
\begin{align}
F_{(z,z^*)} \colon (x,x^*)
\mapsto \iota_{\gra A}(z+x,z^*+x^*)+\langle x,x^*\rangle \label{BrB:PSb1}
\end{align}
is proper, lower semicontinuous, and convex.
Set \begin{align}
G(x,x^*) :=
F_{(z,z^*)}(x,x^*)+ \tfrac{1}{2}\|x\|^2
+\tfrac{1}{2}\|x^*\|^2,\quad \forall (x,x^*)\in X\times X^*.\label{BrB:NPSb1}
\end{align}
Then
\begin{align}
\inf G =-G^*(0,0).
\label{BrB:PS1}
\end{align}
By \eqref{BrB:PSb1},  $\inf G\geq 0$. Then $(0,0)\in\dom G^*$.
By Fact~\ref{f:FJB5},  there exists a sequence
\begin{align}\big((a_n,a^*_n), (y^*_n, y^{**}_n)\big)_{n\in\NN}\;\text{in}\; \gra \partial G
\label{BrB:NPSb2}\end{align}
such that
   \begin{align}(y^*_n, y^{**}_n)\rightarrow (0,0).   \label{coer:1}\end{align}
 Thus,
\begin{align}K_n:=\max\big\{\|y^*_n\|, \|y^{**}_n\|\big\}\rightarrow0.
\label{coer:N1}\end{align}
By Fact~\ref{f:F4} and \eqref{BrB:NPSb2}, there exists $(v^*_n, v^{**}_n)\in Ja_n\times J_{X^*} a^*_n$ such that
\begin{align}(y^*_n, y^{**}_n)\in \partial F_{(z,z^*)}(a_n,a^*_n)+(v^*_n,
v^{**}_n),\quad \forall n \in\NN.
      \label{coer:N2}\end{align}
By \eqref{coer:N2}, \eqref{BrB:PSb1},
and \cite[Theorem 3.2.4(vi)\&(ii)]{Zalinescu},
 there exists a sequence
 $(z^*_n,z^{**}_n)_{n\in\NN}$ in
$(\gra A)^{\bot}$ such that
\begin{align}
(y^*_n, y^{**}_n)=({a^{*}_n},
{a_n})+(z^*_n,z^{**}_n)+(v^*_n, v^{**}_n),\quad \forall n\in\NN.\label{BrB:PS10}
\end{align}
Since $A^*$ is monotone and $(z^{**}_n,z^{*}_n)\in\gra(- A^*)$,
it follows from \eqref{BrB:PS10} that
\begin{align}&\negthinspace\negthinspace\negthinspace\negthinspace\langle
y^*_n, y^{**}_n\rangle+\langle a_n,a^*_n\rangle\nonumber\\ &\quad-\big[\langle y_n^*,
{a_n}\rangle+\langle y_n^{**}, {a^{*}_n}\rangle\big]-\big[\langle
y^*_n, v^{**}_n\rangle+\langle v^*_n,y^{**}_n\rangle\big]\nonumber\\&\quad+\langle a^*_n,v^{**}_n\rangle
+\langle v^*_n, v^{**}_n\rangle+\langle a_n,v^*_n\rangle\nonumber\\
&=\langle y^*_n-a^*_n-v^*_n, y^{**}_n-a_n-v^{**}_n\rangle\nonumber\\
&=\langle z^*_n,z^{**}_n\rangle\leq0, \quad \forall n\in\NN.\label{DPNo:1}\end{align}
Since  $(v^*_n, v^{**}_n)\in Ja_n\times J_{X^*} a^*_n$, by \eqref{DPNo:1}, we have
\begin{align}&\negthinspace\negthinspace\negthinspace\negthinspace\langle y^*_n, y^{**}_n\rangle+\langle a_n,a^*_n\rangle\nonumber\\&-\big[\|y_n^*\|\cdot
\|{a_n}\|+\|y_n^{**}\|\cdot \|{a^{*}_n}\|\big]-\big[
\|y^*_n\|\cdot\|a_n^{*}\|+\|a_n\|\cdot\|y^{**}_n\|\big]\nonumber\\
\quad&+\| a^*_n\|^2
-\|a_n\|\cdot\|a^*_n\|+\|a_n\|^2 \leq0, \quad \forall n\in\NN.\label{BrB:NPS10}\end{align}
Then by \eqref{BrB:NPS10} and \eqref{coer:N1},
\begin{align}&-K^2_n+\langle a_n,a^*_n\rangle-K_n\big[
\|{a_n}\|+\|{a^{*}_n}\|\big]-
K_n\big[\|a_n^{*}\|+\|a_n\|\big]\nonumber\\&\quad+\tfrac{1}{2}\big[\| a^*_n\|^2
+\|a_n\|^2\big] \leq0, \quad \forall n\in\NN.
\label{BrB:NPS11}\end{align}
Hence
\begin{align}\nonumber&-K^2_n+\langle a_n,a^*_n\rangle-2K_n\big[
\|{a_n}\|+\|{a^{*}_n}\|\big]+\tfrac{1}{4}\big[\| a^*_n\|
+\|a_n\|\big]^2 \leq0, \quad \forall n\in\NN.
\label{BrB:NPS12}\end{align}
Set $(x_n,x_n^*):=(z+{a_n},z^*+{a^{*}_n}),\quad \forall n\in\NN$.
Then by \eqref{BrB:PSb1}, we have
\begin{align}
  F_{(z,z^*)}({a_n},
{a^{*}_n})&=\iota_{\gra A}(z+{a_n},z^*
+{a^{*}_n})+\langle {a_n},
{a^{*}_n}\rangle\\&=
  \iota_{\gra A}(x_n,x^*_n)+\langle x_n-z,x^*_n-z^*\rangle.
  \label{TINf:5}
\end{align}
By \eqref{coer:N2} and \eqref{TINf:5},
\begin{align}
(x_n, x^*_n)\in\gra A,\quad \forall n\in\NN.  \label{TINf:6}
\end{align}
Then by \eqref{TINf:6} and   \eqref{TIN:1}, we have
\begin{align}
\langle{a_n},
{a^{*}_n}\rangle=\langle x_n-z,x^*_n-z^*\rangle\geq0,\quad \forall n\in\NN.
  \label{TINf:7}
\end{align}
Combining \eqref{BrB:NPS12} and \eqref{TINf:7},
\begin{align}\tfrac{1}{4}\big(\| a^*_n\|
+\|a_n\|\big)^2 \leq K^2_n+2K_n\big(
\|{a_n}\|+\|{a^{*}_n}\|\big), \quad \forall n\in\NN;
\end{align}
equivalently,
\begin{align}\big(\| a^*_n\|
+\|a_n\|-4K_n\big)^2 \leq 20K^2_n, \quad \forall n\in\NN.
\label{BrB:NPS13}\end{align}
In view of \eqref{coer:N1},
\begin{align}
\|a_n\|+\|a^*_n\|\rightarrow0.\label{TINf:3}
\end{align}
Thus
$(a_n,a^*_n)\to (0,0)$ and hence
$(x_n,x^*_n)\to (z,z^*)$.
Finally, by \eqref{TINf:6} and since $\gra A$ is closed, we see
$(z,z^*)\in\gra A$. Therefore, $A$ is maximally monotone.
\end{proof}

\begin{example}\label{SRE:1}
Let $A:X\rightrightarrows X^*$ be a monotone linear relation such that $\gra A$ is closed.
We note that we cannot guarantee
the maximal monotonicity of $A$ even if $A$ is at most single-valued and
densely defined.
To see this, suppose that  $X=\ell^2$,  and that
$A:\ell^2\rightrightarrows \ell^2$ is given by
\begin{align}Ax:=\frac{\bigg(\sum_{i< n}x_{i}-\sum_{i> n}x_{i}\bigg)_{n\in\NN}}{2}
=\bigg(\sum_{i< n}x_{i}+\tfrac{1}{2}x_n\bigg)_{n\in\NN},
\quad \forall x=(x_n)_{n\in\NN}\in\dom A,\label{EL:1}\end{align}
where $\dom A:=\Big\{ x:=(x_n)_{n\in\NN}\in \ell^{2}\mid \sum_{i\geq 1}x_{i}=0,
 \bigg(\sum_{i\leq n}x_{i}\bigg)_{n\in\NN}\in\ell^2\Big\}$.
 Then $A$ is an at most single-valued linear relation.
Now \cite[Propositions~3.6]{BWY7} states that
\begin{align}
\label{PF:a2}
A^*x= \bigg(\thalb x_n + \sum_{i> n}x_{i}\bigg)_{n\in\NN},
\end{align}
where
\begin{equation*}
x=(x_n)_{n\in\NN}\in\dom A^*=\bigg\{ x=(x_n)_{n\in\NN}\in \ell^{2}\;\; \bigg|\;\;
 \bigg(\sum_{i> n}x_{i}\bigg)_{n\in\NN}\in \ell^{2}\bigg\}.
\end{equation*}
 Moreover,
 \cite[Propositions~3.2, 3.5, 3.6 and 3.8]{BWY7}, \cite[Theorem~2.5]{PheSim} and Fact~\ref{Sv:7}
  show
that:
\begin{enumerate}
 \item\label{NEC:1}
 $A$ is maximally monotone and skew;
\item\label{NEC:2} $\dom A$ is dense and $\dom A\subsetneqq\dom A^*$;
\item\label{NEC:3} $A^*$ is  maximally monotone, but not skew;
\item\label{NEC:4} $-A$ is not maximally monotone.
\end{enumerate}
Hence, $-A$ is  monotone with dense domain and
$\gra (-A)$ is closed,
but nonetheless $-A$ is not maximally monotone. \endproof
\end{example}

\section{The general Br\'ezis-Browder theorem}
\label{s:main}

We may now pack everything together. For ease we repeat Theorem \ref{thm:bbwy}:

\begin{theorem}[Br\'ezis-Browder in general Banach space] \label{BrBrTD:1}
Let $A\colon X\rightrightarrows X^*$ be a monotone linear
 relation  such that $\gra A$ is closed.
Then the following are equivalent.
\begin{enumerate}

\item\label{MaT:1} $A$ is maximally monotone of type (D).

\item \label{MaT:2} $A$ is  maximally monotone  of type (NI).

\item \label{MaT:3} $A$ is  maximally monotone  of type (FP).
\item \label{MaT:03} $A^*$ is monotone.

\end{enumerate}
\end{theorem}

\begin{proof}
Directly combine Fact~\ref{TypeD:1} and Proposition~\ref{PrTD:1}.
\end{proof}

The original Br\'ezis and Browder result follows.

\begin{corollary}[Br\'ezis and Browder]
Suppose that $X$ is  reflexive.
Let $A\colon X \To X^*$ be a monotone linear relation
such that $\gra A$ is closed. Then the following are equivalent.
\begin{enumerate}
\item\label{BRS:1}
$A$ is maximally monotone.
\item\label{BRS:2}
$A^*$ is maximally monotone.
\item\label{BRS:3}
$A^*$ is monotone.
\end{enumerate}
\end{corollary}
\begin{proof}
``\ref{BRS:1}$\Leftrightarrow$\ref{BRS:3}'': Apply Theorem~\ref{BrBrTD:1}
and Fact~\ref{ITPR} directly.

``\ref{BRS:2}$\Rightarrow$\ref{BRS:3}'': Clear.

``\ref{BRS:3}$\Rightarrow$\ref{BRS:2}'': Since $\gra A$ is closed, $(A^*)^*=A$.
Apply Theorem~\ref{BrBrTD:1} to $A^*$  .
\end{proof}

In the case of a skew operator we can add maximality of the adjoint and so we prefigure results of the next section:

\begin{corollary}\label{CoSBr:1}
Let $A\colon X \To X^*$ be a skew operator
such that $\gra A$ is closed. Then the following are equivalent.
\begin{enumerate}
\item\label{MaTS:1} $A$ is maximally monotone of type (D).

\item \label{MaTS:03} $A^*$ is monotone.
\item \label{MaTS:04} $A^*$ is maximally monotone  with respect to $X^{**}\times X^*$.

\end{enumerate}
\end{corollary}
\begin{proof}
By Theorem~\ref{BrBrTD:1}, it only remains to show

``\ref{MaTS:03}$\Rightarrow$\ref{MaTS:04}'':
Let $(z^{**}, z^*)\in X^{**}\times X^*$ be
monotonically related to $\gra A^*$. Since $\gra (-A)\subseteq\gra A^*$,
$(z^{**}, z^*)$ is monotonically related to $\gra (-A)$.
Thus $(z^{*}, z^{**})\in\left[\gra (-A)\right]^{\bot}$ since $\gra A$ is linear. Hence $(z^{**},z^*)\in\gra A^*$.
Hence $A^*$ is maximally monotone.
\end{proof}

\begin{remark}
We cannot say $A^*$ is maximally monotone with respect to
$X^{**}\times X^{***}$ in Corollary~\ref{CoSBr:1}\ref{MaTS:04}:
indeed,
let $A$ be defined by \begin{align*}\gra A=\{0\}\times X^*.\end{align*}
 Then $\gra A^*=\{0\} \times X^*$.
If $X$ is not reflexive, then $X^*\subsetneqq X^{***}$ and so
$\gra A^*$ is a proper subset of $\{0\}\times X^{***}$. Hence
$A^*$ is not maximally monotone with respect to $X^{**}\times X^{***}$ although
$A$ is maximally monotone of type (D) (since $A=N_{\{0\}}$
by Fact~\ref{ITPR}).
\end{remark}

In the next section,
we turn to the question of how the skew part of the adjoint behaves.

\section{Decomposition of monotone linear relations}\label{s:skew}
Let us first gather some basic properties about monotone linear relations,
and conditions for them to be maximally monotone.

The next three propositions were
 proven in reflexive spaces in \cite[Proposition~2.2]{BWY3}.
We adjust the proofs to a general Banach space setting.

\begin{proposition}[Monotone linear relations]\label{P:a}
Let $A\colon X \To X^*$ be a linear relation.
Then the following hold.
\begin{enumerate}
\item\label{Nov:s1}
Suppose $A$ is monotone. Then
$\dom A\subseteq (A0)_{\bot}$ and $A0\subseteq (\dom A)^\bot$;
consequently, if $\gra A$ is closed, then
$\dom A\subseteq\overline{\dom A^*}^{\wk}\cap X$ and $A0\subseteq A^*0$.
\item \label{Th:032}
$(\forall x\in\dom A)(\forall z\in (A0)_{\bot})$
$\langle z,Ax\rangle$ is single-valued.
\item \label{Sia:2c}
$(\forall z \in (A0)_{\bot})$ $\dom A \to \RR \colon y\mapsto \scal{z}{Ay}$
is linear.
\item\label{Th:031} If $A$ is monotone, then
$(\forall x\in\dom A)$ $\langle x,Ax\rangle$ is single-valued.
\item\label{Th:031L}$A$ is monotone $\Leftrightarrow$
$(\forall x\in\dom A)$ $\inf\langle x,Ax\rangle$ $\geq 0$.
\item\label{Th:031+} If $(x,x^*)\in (\dom A)\times X^*$ is monotonically
related to $\gra A$ and $x_{0}^*$ $\in Ax$, then
$x^*-x_{0}^*\in (\dom A)^{\perp}.$

\end{enumerate}
\end{proposition}

\begin{proof}
\ref{Nov:s1}: Pick $x\in \dom A$.
Then there exists $x^*\in X^*$ such that
  $(x,x^*)\in \gra A.$  By
  monotonicity of $A$ and since $\{0\}\times A0\subseteq \gra A$,
 we have  $\langle x, x^*\rangle\geq \sup\langle x,A0\rangle$.
Since $A0$ is a linear subspace,
we obtain
$x\bot A0$. This implies
$\dom A\subseteq (A0)_{\bot}$ and $A0\subseteq (\dom A)^\bot$.
If $\gra A$ is closed, then Fact~\ref{Rea:1}\ref{Sia:2}\&\ref{Sia:1}
yields $\dom A\subseteq (A0)_{\bot}
\subseteq (A0)^{\bot}=\overline{\dom A^*}^{\wk}$ and  $A0\subseteq A^*0$.

\ref{Th:032}:
Take $x\in\dom A$, $x^*\in Ax$, and $z\in (A0)_{\bot}$.
By Fact~\ref{Rea:1}\ref{Th:28},
$\langle z,Ax\rangle=\langle z,x^*+A0\rangle=\langle z,x^*\rangle$.

\ref{Sia:2c}:
Take $z\in (A0)_{\bot}$. By \ref{Th:032},
$(\forall y\in\dom A)$ $\langle z, Ay\rangle$
is single-valued.
Now let $x,y$ be in $\dom A$, and let $\alpha,\beta$ be in $\RR$.
If $(\alpha,\beta) = (0,0)$, then
$\langle z, A(\alpha x+\beta y)\rangle=\langle z, A0\rangle=0
=\alpha \langle z, Ax\rangle +\beta \langle z, Ay\rangle$.
And if $(\alpha,\beta)\neq (0,0)$, then
Fact~\ref{Rea:1}\ref{Th:30} yields
$\langle z, A(\alpha x+\beta y)\rangle =\langle z, \alpha Ax +\beta Ay\rangle
=\alpha \langle z, Ax\rangle +\beta\langle z, Ay\rangle$.
This verifies  linearity.

\ref{Th:031}: Apply
\ref{Nov:s1}\&\ref{Th:032}.

\ref{Th:031L}: ``$\Rightarrow$'': This  follows from
 the fact that $(0,0)\in\gra A$.
``$\Leftarrow$'': If $x$ and $y$ belong to $\dom A$, then
Fact~\ref{Rea:1}\ref{Th:30} yields
$\langle x-y,Ax-Ay\rangle=\langle x-y,A(x-y)\rangle\geq 0$.

\ref{Th:031+}: Let $(x,x^*)\in\dom A\times X^*$
be monotonically related to $\gra A$,
and take $x_0^*\in Ax.$
For every $(v,v^*)\in\gra A$, we have $x_0^*+
v^*\in A(x+v)$ (by Fact~\ref{Rea:1}\ref{Th:30}); hence,
$\langle x-(x+v), x^*-(x_0^*+ v^*)\rangle\geq 0$ and thus
$\langle v,  v^*\rangle\geq \langle v, x^*-x_0^*
\rangle$. Now take
$\lambda>0$  and replace $(v,v^*)$ in the last inequality
by $(\lambda v,\lambda v^*)$.
Then divide by $\lambda$ and let $\lambda\rightarrow 0^+$ to see that
$0\geq \sup\langle \dom A,\ x^*-x_0^*\rangle$.
Since $\dom A$ is linear, it follows that  $x^*-x_0^*\in
 (\dom A)^\bot$.
\end{proof}

We define the \emph{symmetric part} and the \emph{skew part} of $A$ via
\begin{equation}
\label{Fee:1}
A_+ := \thalb A + \thalb A^* \quad\text{and}\quad
A_{\mathlarger{\circ}} := \thalb A - \thalb A^*,
\end{equation}
respectively. It is easy to check that $A_+$ is symmetric and that $A_{\mathlarger{\circ}} $
is skew.

\begin{proposition}[Maximally monotone linear relations]\label{linear}
Let $A\colon X \To X^*$ be a monotone linear relation.
Then the following hold.
\begin{enumerate}
\item\label{sia:2v} If $A$ is maximally monotone,
 then $(\dom A)^{\perp}= A0$ and hence
$\overline{\dom A}=(A0)_\bot$.
\item \label{sia:3v}
If $\dom A$ is closed, then:
$A$ is maximally monotone $\Leftrightarrow$ $(\dom A)^\bot = A0$.
\item \label{sia:3iv}
If $A$ is maximally monotone, then
$\overline{\dom A^*}^{\wk}\cap X=\overline{\dom A}=(A0)_\bot$,
 and
$A0=A^*0=A_+0 = A_{\mathlarger{\circ}}0= (\dom A)^{\perp}$ is
(weak$^*$) closed.
\item \label{sia:3vi}
If $A$ is maximally monotone and $\dom A$ is closed,
then $\dom A^*\cap X=\dom A$.

\item \label{sia:3vii}
If $A$ is maximally monotone and $\dom A \subseteq \dom A^*$,
then
$A=A_{+}+A_{\mathlarger{\circ}}$, $A_{+}=A-A_{\mathlarger{\circ}}$,
 and $A_{\mathlarger{\circ}}=A-A_{+}$.

\item \label{dmns:1}
If $A$ is maximally monotone and $\dom A$ is closed,
then both $A_+$ and $A_{\mathlarger{\circ}}$ are maximally monotone.

\item \label{dmns:2}
If $A$ is maximally monotone and $\dom A$ is closed,
then $A^*=(A_+)^*+(A_{\mathlarger{\circ}})^*$.
\end{enumerate}
\end{proposition}

\begin{proof}
\ref{sia:2v}: Since
$A+N_{\dom A}=A+(\dom A)^{\perp}$
is a monotone extension of $A$ and $A$ is
maximally monotone, we must have
$A+(\dom A)^{\perp}=A$. Then $A0+(\dom A)^{\perp}=A0$. As $0\in A0$,
$(\dom A)^{\perp}\subseteq A0.$ Combining with
Proposition~\ref{P:a}\ref{Nov:s1},  we have
$(\dom A)^{\bot}=A0$. By Fact~\ref{rudin:1},
$\overline{\dom A}=(A0)_{\bot}$.

 \ref{sia:3v}:  ``$\Rightarrow$": Clear from \ref{sia:2v}.
 ``$\Leftarrow$":
 The assumptions and
 Fact~\ref{rudin:1} imply that  $\dom A=\overline{\dom A}=
\left[ (\dom A)^{\bot}\right]_{\bot}=(A0)_{\bot}$.
 Let $(x,x^*)$ be monotonically related to $\gra A$.
   We have $\inf\langle x-0, x^*-A0\rangle\geq 0$.
 Then we have $x\in (A0)_{\bot}$ and hence $x\in\dom A$.
  Then by Proposition~\ref{P:a}\ref{Th:031+}
 and Fact~\ref{Rea:1}\ref{Th:28}, $x^*\in Ax$. Hence $A$ is maximally monotone.

 \ref{sia:3iv}: By \ref{sia:2v} and  Fact~\ref{Rea:1}\ref{Sia:1},
 $A0=(\dom A)^{\perp}=A^*0$ is weak$^*$ closed and thus
 $A_+0 = A_{\mathlarger{\circ}}0=A0=(\dom A)^{\perp}$.
 Then by Fact~\ref{Rea:1}\ref{Sia:2} and \ref{sia:2v},
  $\overline{\dom A^*}^{\wk}\cap X=(A0)_{\bot}=\overline{\dom A}$.

 \ref{sia:3vi}: Combine  \ref{sia:3iv} with Fact~\ref{Rea:1}\ref{Th:32}.

 \ref{sia:3vii}:
  We show only the proof of $A=A_{+}+A_{\mathlarger{\circ}}$
as the other two proofs are analogous.
Clearly, $\dom A_{+}=\dom A_{\mathlarger{\circ}}
=\dom A\cap\dom A^*=\dom A$.
Let $x\in\dom A$, and $x^*\in Ax$ and $y^*\in A^*x$. We write
$x^*=\tfrac{x^*+y^*}{2}+\tfrac{x^*-y^*}{2}\in (A_{+}+A_{\mathlarger{\circ}})x$.
Then, by \ref{sia:3iv} and Fact~\ref{Rea:1}\ref{Th:28},
$Ax=x^*+A0=x^*+(A_{+}+A_{\mathlarger{\circ}})0=
(A_{+}+A_{\mathlarger{\circ}})x$.
Therefore, $A=A_{+}+A_{\mathlarger{\circ}}$.

\ref{dmns:1}:
By \ref{sia:3vi},
\begin{align}\dom A_+=\dom{A_\mathlarger{\circ}}=\dom A\ \text{ is closed}.\label{SCK:05}
\end{align}
Hence, by  \ref{sia:3iv},
\begin{align}
A_{\mathlarger{\circ}}0=A_+0=A0=(\dom A)^{\bot}=(\dom A_+)^{\bot}
= (\dom A_{\mathlarger{\circ}})^{\bot}.\label{SCK:M6}
\end{align}
Since $A$ is monotone,
so are  $A_+$ and $A_\mathlarger{\circ}$.
Combining \eqref{SCK:05}, \eqref{SCK:M6}, and  \ref{sia:3v},
we deduce that
$A_+$ and  $A_{\mathlarger{\circ}}$ are maximally monotone.

\ref{dmns:2}: By \ref{sia:3vi}\&\ref{sia:3vii},
\begin{align}
A=A_++ A_{\mathlarger{\circ}}.
\end{align}
Then  by \ref{dmns:1}, \ref{sia:3vi}, and Fact~\ref{LinAdSum},
$A^*=(A_+)^*+ (A_{\mathlarger{\circ}})^*$.
\end{proof}

For a monotone linear relation $A\colon X \To X^*$ it will be convenient
to define --- as in, e.g., \cite{BBW} --- a generalized quadratic form
\begin{equation*}
(\forall x\in X)\quad
q_A(x) =  \begin{cases} \tfrac{1}{2}\langle x,Ax\rangle,&\text{if $x\in \dom A$};\\
+\infty,&\text{otherwise}.\end{cases}\end{equation*}
We  write $\overline{q_A}$ for the lower semicontinuous hull of $q_A$.

\begin{proposition}\label{Nov:s2} Let $A\colon X \To X^*$ be a monotone linear
relation,
let $x$ and $y$ be in $\dom A$, and let $\lambda\in\RR$.
Then
$q_A$ is single-valued, $q_A\geq0$ and
\begin{align}
&\lambda q_A(x) + (1-\lambda)q_A(y) - q_{A}(\lambda x + (1-\lambda)y) =
\lambda(1-\lambda)q_A(x-y)\nonumber\\& =
\tfrac{1}{2}\lambda(1-\lambda)\scal{x-y}{Ax-Ay}.\label{deli:1}
\end{align}
Consequently, $q_A$ is convex.
\end{proposition}
\begin{proof}
Proposition~\ref{P:a}\ref{Th:031}\&\ref{Th:031L} show
that $q_A$ is single-valued and that $q_A\geq0$.
Combining with Proposition~\ref{P:a}\ref{Nov:s1}\&\ref{Sia:2c}, we obtain
\eqref{deli:1}.
Therefore, $q_A$ is convex.
\end{proof}

As in the classical case, $q_A$ allows us to connect properties of $A_+$ to those of $A$ and $A^*$.

\begin{proposition} \label{f:PheSim} Let  $A:X\rightrightarrows X^*$
 be a monotone  linear relation.
Then the following hold.
\begin{enumerate}
\item \label{f:PheSim:lsc02}  $\overline{q_A}+\iota_{\dom A_+}=q_{A_+}$ and thus $q_{A_+}$ is convex.
\item \label{f:PheSim:quad}
$\gra A_+ \subseteq\gra \partial \overline{q_A}$.
If $A_+$ is maximally monotone, then $A_+=\partial \overline{q_A}$.

\item \label{dms:quad}
If $A$ is maximally monotone and $\dom A$ is closed,
then $A_+=\partial \overline{q_A}$.
\item \label{f:PheSim:brah:1}If $A$ is maximally monotone,
then $A^*|_X$ is monotone.
\item \label{f:PheSim:brah:2}If $A$ is maximally
monotone and $\dom A$ is closed, then $A^*|_X$ is maximally monotone.
\end{enumerate}
\end{proposition}
\begin{proof}
Let $x\in\dom A_+$.

\ref{f:PheSim:lsc02}: By Fact~\ref{Rea:1}\ref{Sia:2b} and Proposition~\ref{P:a}\ref{Th:031},
 $q_{A_+}=q_{A}|_{\dom A_+}$.
 Then by Proposition~\ref{Nov:s2},
  $q_{A_+}$ is convex. Let $y\in\dom A$. Then by
  Fact~\ref{Rea:1}\ref{Sia:2b},
\begin{align}0\leq\tfrac{1}{2}\langle Ax-Ay,x-y\rangle=
\tfrac{1}{2}\langle Ay,y\rangle+\tfrac{1}{2}\langle Ax,x\rangle-\langle A_+x, y\rangle,
\label{Bor:u2}\end{align}
we have $q_A(y)\geq\langle A_+x, y\rangle-q_A(x)$.
 Take the lower semicontinuous hull of $q_A$ at $y$ to
 deduce that $\overline{q_A}(y)\geq\langle A_+x, y\rangle-q_A(x)$.
  For $y=x$, we have $\overline{q_A}(x)\geq q_A(x)$.
 On the other hand,
  $\overline{q_A}\leq q_A$.
Altogether, $\overline{q_A}(x)=q_A(x)=q_{A_+}(x)$.
Thus \ref{f:PheSim:lsc02} holds.

\ref{f:PheSim:quad}:
Let $y\in\dom A$. By \eqref{Bor:u2} and \ref{f:PheSim:lsc02},
\begin{align}q_A(y)\geq
q_A(x)+\langle A_+x,y-x\rangle=\overline{q_A}(x)
+\langle A_+x,y-x\rangle.\label{Bor:u3}
\end{align}
Since $\dom\overline{q_A}\subseteq\overline{\dom q_A}
=\overline{\dom A}$, by \eqref{Bor:u3},
 $\overline{q_A}(z)\geq
\overline{q_A}(x)+\langle A_+x,z-x\rangle,\quad \forall z\in\dom\overline{q_A}.$
Hence $A_+x\subseteq\partial \overline{q_A}(x)$.
If $A_+$ is maximally monotone, then  $A_+=\partial \overline{q_A}$.
Thus \ref{f:PheSim:quad} holds.

\ref{dms:quad}: Combine Proposition~\ref{linear}\ref{dmns:1}
with \ref{f:PheSim:quad}.

\ref{f:PheSim:brah:1}:
Suppose to the contrary that $A^*|_X$ is not monotone.
By Proposition~\ref{P:a}\ref{Th:031L},
there exists $(x_0,x_0^*)\in\gra A^*$ with $x_0\in X$ such that
 $\langle x_0,x_0^*\rangle<0$.
Now we have
\begin{align}&\langle-x_0-y,x^*_0-y^*\rangle
=-\langle x_0,x^*_0\rangle+\langle y,y^*\rangle+\langle x_0,y^*\rangle
-\langle y,x_0^*\rangle\nonumber\\
&=-\langle x_0,x^*_0\rangle+\langle y,y^*\rangle>0,
\quad \forall (y,y^*)\in\gra A.\label{Brezi:1}\end{align}
Thus, $(-x_0,x^*_0)$ is monotonically related to $\gra A$.
 By maximal monotonicity of $A$, $(-x_0,x^*_0)\in\gra A$.
Then $\langle-x_0-(-x_0),x^*_0-x_0^*\rangle=0$,
 which contradicts \eqref{Brezi:1}. Hence $A^*|_X$ is monotone.

\ref{f:PheSim:brah:2}: By Fact~\ref{Rea:1}\ref{Th:32},
$\dom A^*|_X=(A0)_{\bot}$ and thus $\dom A^*|_X$ is closed.
By Fact~\ref{rudin:1} and Proposition~\ref{linear}\ref{sia:2v},
$(\dom A^*|_X)^{\bot}=((A0)_{\bot})^{\bot}
=\overline{A0}^{\wk}=A0$.
Then by Proposition~\ref{linear}\ref{sia:3iv},
$(\dom A^*|_X)^{\bot}=A^*0=A^*|_X 0$.
Applying
\ref{f:PheSim:brah:1} and Proposition~\ref{linear}\ref{sia:3v},
we see that
 $A^*|_X$ is maximally monotone.
\end{proof}

The proof of Proposition~\ref{f:PheSim}\ref{f:PheSim:brah:1}
was borrowed from \cite[Theorem~2]{Brezis-Browder}.
Results very similar to
Proposition~\ref{f:PheSim}\ref{f:PheSim:lsc02}\&\ref{f:PheSim:quad}
are verified in \cite[Proposition~18.9]{Yao}.
The proof of the next Theorem~\ref{TSkeD:1}\ref{MaTC:1}$\Rightarrow$\ref{MaTC:02}
was partially inspired by that of \cite[Theorem~4.1(v)$\Rightarrow$(vi)]{BB}.
When the domain of $A$ is closed we can obtain additional information
 about the skew part of $A$.

\begin{theorem}[Monotone relations with closed graph and domain]\label{TSkeD:1}
Let $A:X\rightrightarrows X^*$ be a monotone linear relation
 such that $\gra A$ is closed and $\dom A$ is closed.
 Then the following are equivalent.\begin{enumerate}
\item\label{MaTC:1} $A$ is maximally monotone of type (D).

\item \label{MaTC:02} $A_{\mathlarger{\circ}}$
is maximally monotone  of type (D)  with respect to $X\times X^*$ and $A^*0=A0$.
\item \label{MaTC:2} $(A_{\mathlarger{\circ}})^*$
is maximally monotone  with respect to $X^{**}\times X^*$ and $A^*0=A0$.
\item \label{MaTC:3} $(A_{\mathlarger{\circ}})^*$ is  monotone and $A^*0=A0$.
\item \label{MaTC:4} $A^*$ is  monotone.
\item \label{MaTC:5} $A^*$ is  maximally monotone with respect to $X^{**}\times X^*$.
\end{enumerate}
\end{theorem}
\begin{proof}
``\ref{MaTC:1}$\Rightarrow$\ref{MaTC:02}": By Fact~\ref{TypeD:1},
\begin{align}
A^*\ \text{is monotone}.\label{SCK:1}
\end{align}
By Proposition~\ref{f:PheSim}\ref{dms:quad} and Fact~\ref{ITPR},
 \begin{align}
A_+\
\text{ is maximally monotone of type (D)}.\label{SCK:04}
\end{align}
By Fact~\ref{TypeD:1},
\begin{align}
(A_+)^* \ \text{is monotone}.\label{SCK:5a}
\end{align}
Now we show that
\begin{align}
(A_{\mathlarger{\circ}})^* \ \text{is monotone}.\label{SCK:5}
\end{align}
Proposition~\ref{linear}\ref{dmns:2} implies
\begin{align}
A^*=(A_+)^*+(A_{\mathlarger{\circ}})^*.\label{SCK:3}
\end{align}
Since $A$ is maximally monotone and $\dom A$ is closed,
Proposition~\ref{linear}\ref{dmns:1} implies that
$A_{\mathlarger{\circ}}$ is maximally monotone.
Hence $\gra (A_{\mathlarger{\circ}})$ is closed.
On the other hand,
again since $A$ is maximally monotone and $\dom A$ is closed,
Proposition~\ref{linear}\ref{sia:3vi} yields
$\dom (A_{\mathlarger{\circ}}) = \dom A$ is closed.
Altogether, and combining with Fact~\ref{Rea:1}\ref{Th:32}
applied to $A_{\mathlarger{\circ}}$, we obtain
$\dom (A_{\mathlarger{\circ}})^* = (A_{\mathlarger{\circ}}0)^{\bot}$.
Furthermore, since $A0 = A_{\mathlarger{\circ}}0$ by
Proposition~\ref{linear}\ref{sia:3iv}, we have
$(A0)^\bot = (A_{\mathlarger{\circ}}0)^\bot$.
Moreover, applying Fact~\ref{Rea:1}\ref{Th:32}
to $A$, we deduce that $\dom A^* = (A0)^\bot$.
Therefore,
 \begin{align}\dom (A_{\mathlarger{\circ}})^*
 = (A_{\mathlarger{\circ}}0)^{\bot}=(A0)^{\bot}=\dom A^*.
 \label{SCK:6}
 \end{align}
 Similarly, we have
  \begin{align}\dom (A_+)^*=\dom A^*.
 \label{SCK:7}
 \end{align}
Take $(x^{**},x^*)\in\gra (A_{\mathlarger{\circ}})^*$.
By \eqref{SCK:3} and \eqref{SCK:6},
there exist $a^*, b^*\in X^*$ such that \begin{align}
(x^{**},a^*)\in \gra A^*,
 (x^{**}, b^*)\in \gra (A_+)^*\label{SCK:08}\end{align} and
\begin{align}
a^*=b^*+x^*.\label{SCKL:08}
\end{align}
Since $A_+$ is symmetric, $\gra A_+\subseteq\gra (A_+)^*$.
Thus, by \eqref{SCK:5a},
$(x^{**}, b^*)$ is monotonically related to $\gra A_+$.
 By \eqref{SCK:04}, there exist
a  bounded net
$(a_{\alpha}, b^*_{\alpha})_{\alpha\in\Gamma}$ in $\gra A_+$
such that
$(a_{\alpha}, b^*_{\alpha})_{\alpha\in\Gamma}$
weak*$\times$strong converges to
$(x^{**},b^*)$.
Thus $(a_{\alpha}, b^*_{\alpha})\in\gra (A_+)^*$.
By \eqref{SCK:7} and \eqref{SCK:3},
there exist  $a^*_{\alpha}\in A^* a_{\alpha},
 c^*_{\alpha}\in (A_{\mathlarger{\circ}})^*a_{\alpha}$ such that
\begin{align}
a^*_{\alpha}=b^*_{\alpha}+c^*_{\alpha},\quad\forall \alpha\in\Gamma.
 \label{SCK:8}
\end{align}
Thus by Fact~\ref{Rea:1}\ref{Sia:2b},
\begin{align}
\langle a_{\alpha}, c^*_{\alpha}\rangle
=\langle A_{\mathlarger{\circ}}a_{\alpha}, a_{\alpha}\rangle=0,
\quad  \forall \alpha\in\Gamma.
\label{SCK:9}
\end{align}
Hence for every $\alpha\in\Gamma$,
$(-a_{\alpha}, c^*_{\alpha}) $ is monotonically
 related to $\gra A_{\mathlarger{\circ}}$.
By Proposition~\ref{linear}\ref{dmns:1},
\begin{align}
(-a_{\alpha}, c^*_{\alpha})\in\gra A_{\mathlarger{\circ}},\quad
\forall \alpha\in\Gamma.\label{SCK:9a}
\end{align}
By \eqref{SCK:1} and \eqref{SCK:08}, we have
\begin{align}
0&\leq\langle x^{**}-a_{\alpha}, a^*-a^*_{\alpha}\rangle
=\langle x^{**}-a_{\alpha}, a^*-b^*_{\alpha}-c^*_{\alpha}\rangle
\quad\text{(by \eqref{SCK:8})} \nonumber\\
&=\langle x^{**}-a_{\alpha}, a^*-b^*_{\alpha}\rangle-\langle x^{**}, c^*_{\alpha}\rangle
+\langle a_{\alpha}, c^*_{\alpha}\rangle\nonumber\\
&=\langle x^{**}-a_{\alpha}, a^*-b^*_{\alpha}\rangle-\langle x^{**}, c^*_{\alpha}\rangle
\quad\text{(by \eqref{SCK:9})}\nonumber\\
&=\langle x^{**}-a_{\alpha}, a^*-b^*_{\alpha}\rangle+\langle x^{*}, a_{\alpha}\rangle
\quad \text{(by \eqref{SCK:9a} and $(x^{**},x^*)\in\gra (A_{\mathlarger{\circ}})^*$)}.
\label{SCK:10}
\end{align}
Taking  the limit in \eqref{SCK:10} along with
 $ a_{\alpha} \weakstarly x^{**}$ and $b^*_{\alpha}
\rightarrow b^*$, we have
\begin{align*}
\langle x^{**},x^*\rangle\geq0.
\end{align*}
Hence $(A_{\mathlarger{\circ}})^*$ is monotone
and thus \eqref{SCK:5} holds.  Combining \eqref{SCK:5}, Proposition~\ref{linear}\ref{dmns:1}
and Fact~\ref{TypeD:1}, we see that
$A_{\mathlarger{\circ}}$ is of type (D).

``\ref{MaTC:02}$\Rightarrow$\ref{MaTC:2}$\Rightarrow$\ref{MaTC:3}":
 Apply Corollary~\ref{CoSBr:1} to
$A_{\mathlarger{\circ}}$.

``\ref{MaTC:3}$\Rightarrow$\ref{MaTC:4}":
By Fact~\ref{Rea:1}\ref{Sia:1} and Proposition~\ref{linear}\ref{sia:3v},
$A$ is maximally monotone.
Then by Proposition~\ref{linear}\ref{dmns:2}
and Proposition~\ref{f:PheSim}\ref{dms:quad}, we have
\begin{align}
A^*=(A_+)^*+(A_{\mathlarger{\circ}})^*\ \text{and}
 \ A_+=\partial \overline{q_A}.\label{SCK:11}
\end{align}
Then $A_+$ is of type (D) by Fact~\ref{ITPR}, and hence $(A_+)^*$
 is monotone by Fact~\ref{TypeD:1}.
Thus, by the assumption and \eqref{SCK:11}, we have $A^*$ is monotone.

``\ref{MaTC:4}$\Rightarrow$\ref{MaTC:5}":
By Proposition~\ref{PrTD:1},
$A$ is maximally monotone. Then by Fact~\ref{Rea:1}\ref{Th:32}
and Proposition~\ref{linear}\ref{sia:3iv},
\begin{align}
\dom A^*=(A^*0)^{\bot}.\label{MAJo:1}
\end{align}
Then by Fact~\ref{rudin:1} and Fact~\ref{Rea:1}\ref{Sia:1},
\begin{align}
\left[\dom A^*\right]_{\bot}=A^*0.\label{MAJo:2}
\end{align}
Let $(x^{**},x^*)\in X^{**}\times X^*$
be monotonically related to $\gra A^*$.
Because $\{0\}\times A^*0\subseteq\gra A^*$,
we have $\inf\langle x^{**}, x^*-A^*0\rangle\geq 0$.
 Since $A^*0$ is a subspace, $x^{**}\in (A^*0)^{\bot}$.
Then by \eqref{MAJo:1},
\begin{align}x^{**}\in\dom A^*.\end{align}
Take $(x^{**},x_0^{**})\in \gra A^{*}$ and
$\lambda>0$.
For every $(a^{**},a^*)\in\gra A^{*}$, we have $(\lambda a^{**}, \lambda a^*)\in\gra A^{*}$ and
hence
 $(x^{**}+\lambda a^{**},x_0^*+
\lambda a^*)\in \gra A^{*}$ (since $\gra A^{*}$ is a subspace).
Thus
\begin{align*}\lambda\langle a^{**}, x_0^*+\lambda a^*-x^*\rangle=
\langle x^{**}+\lambda a^{**}-x^{**}, x_0^*+\lambda a^*-x^*\rangle\geq 0.\end{align*}
Now divide by $\lambda$ to obtain
$\lambda\langle a^{**},  a^*\rangle\geq \langle a^{**}, x^*-x_0^*
\rangle$.
Then  let $\lambda\rightarrow 0^+$ to see that
$0\geq \sup\langle \dom A^*,\ x^*-x_0^*\rangle$.  Thus, $x^*-x_0^*\in (\dom A^*)_{\bot}$.
By \eqref{MAJo:2}, $x^*\in x^*_0+A^*0\subseteq A^*x^{**}+A^*0$. Then there exists $(0,z^*)\in \gra A^*$ such that $(x^{**}, x^*-z^*)\in\gra A^*$. Since $\gra A^*$ is s a subspace,
 $(x^{**}, x^*)=(0,z^*)+(x^{**}, x^*-z^*)\in\gra A^*$.
Hence
$A^*$ is maximally monotone with respect to $X^{**}\times X^*$.

``\ref{MaTC:5}$\Rightarrow$\ref{MaTC:1}":
  Apply Proposition~\ref{PrTD:1} directly.
\end{proof}

The next three examples show the need for various of our auxiliary hypotheses.

\begin{example}
We cannot remove the condition that $A^*0=A0$ in
 Theorem~\ref{TSkeD:1}\ref{MaTC:3}. For example,
  suppose that $X=\RR^2$  and set $e_1=(1,0), e_2=(0,1)$. We define
  $A:X\rightrightarrows X$ by
\begin{equation*}
\gra A=\spand\{e_1\}\times\{0\}\;\text{ so that }\;
\gra A^*=X\times \spand\{e_2\}.
\end{equation*}
Then $A$ is monotone, $\dom A$ is closed, and $\gra A$ is closed.
Thus
\begin{align}
\gra A_{\mathlarger{\circ}}=\spand\{e_1\}\times\spand\{e_2\}
\end{align}
and so
\begin{align*}\gra (A_{\mathlarger{\circ}})^*=\spand\{e_2\}\times\spand\{e_1\}.
\end{align*} Hence  $(A_{\mathlarger{\circ}})^*$ is monotone,
but $A$ is not maximally monotone because $\gra A \subsetneqq \gra N_X$.
\endproof
\end{example}

\begin{example}\label{CDEx:1}
We cannot replace that ``$\dom A$ is closed"
 by  that ``$ \dom A$ is dense" in the statement of
 Theorem~\ref{TSkeD:1}.
 For example, let $X, A$ be defined as in Example~\ref{SRE:1}
 and consider the operator $A^*$.
Example~\ref{SRE:1}\ref{NEC:3}\&\ref{NEC:2} state that  $A^*$ is maximally
monotone with dense domain; hence, $\gra A^*$ is closed.
 Moreover, by Example~\ref{SRE:1}\ref{NEC:1},
 \begin{align}
 (A^*)_{\mathlarger{\circ}}=-A.
 \end{align}
Hence
 \begin{align}
 \left[(A^*)_{\mathlarger{\circ}}\right]^*=-A^*.
 \end{align}
Thus $\left[(A^*)_{\mathlarger{\circ}}\right]^*$ is
 not  monotone by Example~\ref{SRE:1}\ref{NEC:3};
 even though $A^*$ is a classically maximally monotone and densely defined linear operator. \endproof
\end{example}

\begin{example}
We cannot remove the condition that
$(A_{\mathlarger{\circ}})^*$ is monotone
in Theorem~\ref{TSkeD:1}\ref{MaTC:3}. For example,
consider the Gossez operator $A$ (see \cite{Gossez72} and \cite{BB}).
It satisfies
$X=\ell^1$, $\dom A=X$,
$A_{\mathlarger{\circ}} = A$,
$A0=\{0\}=A^*0$, yet $A^*$ is not monotone.\endproof
\end{example}

\begin{remark}
Let $A\colon X\rightrightarrows X^*$
 be a maximally monotone linear relation.
 \begin{enumerate}
 \item
 In general, $(A^*)_{\mathlarger{\circ}}\neq (A_{\mathlarger{\circ}})^*$.
 To see that, let $X, A$ be  as in Example~\ref{SRE:1} again.
  By Example~\ref{SRE:1}\ref{NEC:1},
 we have
\begin{align*}(A^*)_{\mathlarger{\circ}}=-A\;\text{and}\;
 (A_{\mathlarger{\circ}})^*= A^*.
\end{align*}
Hence $(A^*)_{\mathlarger{\circ}}\neq (A_{\mathlarger{\circ}})^*$ by Example~\ref{SRE:1}\ref{NEC:2}.

\item
However, if $X$ is finite-dimensional, we do have $(A^*)_{\mathlarger{\circ}}= (A_{\mathlarger{\circ}})^*$.
Indeed, by Fact~\ref{LinAdSum},
\begin{align*}
(A_{\mathlarger{\circ}})^*=\left(\frac{A-A^{*}}{2}\right)^*
=\frac{A^{*}-A^{**}}{2}=(A^*)_{\mathlarger{\circ}}.
\end{align*}
We expect that  $(A^*)_{\mathlarger{\circ}}= (A_{\mathlarger{\circ}})^*$ for all maximally  monotone linear relations if and only if $X$ is 
finite-dimensional.
\end{enumerate}
\end{remark}

We are now able to present our main result relating  monotonicity and adjoint properties of $A$ and those of its skew part
$A_{\mathlarger{\circ}}$.

\begin{theorem} [Adjoint characterizations of type (D)] \label{BrBrTDS:1}
Let $A\colon X\rightrightarrows X^*$
 be a monotone linear relation such that $\gra A$ is closed and
$\dom A$ is closed.
Then the following are equivalent.
\begin{enumerate}

\item $A$ is maximally monotone of type (D).

\item $A$ is  maximally monotone  of type (NI).

\item $A$ is  maximally monotone  of type (FP).
\item  $A^*$ is monotone.
\item  $A^*$  is maximally monotone with respect to $X^{**}\times X^*$.
\item $A_{\mathlarger{\circ}}$
is maximally monotone  of type (D) and $A^*0=A0$.
\item $(A_{\mathlarger{\circ}})^*$
is maximally monotone with respect to $X^{**}\times X^*$ and $A^*0=A0$.
\item $(A_{\mathlarger{\circ}})^*$ is  monotone and $A^*0=A0$.

\end{enumerate}
\end{theorem}

\begin{proof}
Apply Theorem~\ref{TSkeD:1} and Theorem~\ref{BrBrTD:1}.
\end{proof}

The work in \cite{BBWY3} suggests that in every nonreflexive Banach space there is a maximally monotone linear relation which is not of type (D).

When  $A$   is linear and continuous, Theorem~\ref{BrBrTDS:1}
can also be deduced from \cite[Theorem~4.1]{BB}.
When $X$ is reflexive and $\dom A$ is closed,
Theorem~\ref{BrBrTDS:1} turns into the following
refined version of Fact~\ref{Sv:7}:

\begin{corollary} \label{BrBrTDS:2}
Suppose that $X$ is reflexive and let
$A\colon X\rightrightarrows X^*$
 be a monotone linear relation.
 such that $\gra A$ is closed and $\dom A$ is closed.
Then the following are equivalent.
\begin{enumerate}
\item
$A$ is  maximally monotone.
\item
$A^*$ is monotone.
\item
$A^*$ is maximally monotone.
\item
$A0=A^*0$.
\end{enumerate}
\end{corollary}
\begin{proof}
``(i)$\Leftrightarrow$(ii)$\Leftrightarrow$(iii)$\Rightarrow$(iv)'':
This follows from
Theorem~\ref{BrBrTDS:1} and
Fact~\ref{ITPR}(ii).

``(iv)$\Rightarrow$(i)'':
Fact~\ref{Rea:1}\ref{Sia:1} implies that
$(\dom A)^\perp = A^*0 = A0$.
By Proposition~\ref{linear}\ref{sia:3v}, $A$ is maximally monotone.
\end{proof}

When $X$ is finite-dimensional, the closure assumptions in
the previous result are automatically satisfied
and we thus obtain the following:

\begin{corollary}
Suppose that $X$ is finite-dimensional.
Let $A\colon X\rightrightarrows X^*$
 be a monotone linear relation.
Then the following are equivalent.
\begin{enumerate}
\item
$A$ is  maximally monotone.
\item
$A^*$ is monotone.
\item
$A^*$ is maximally monotone.
\item
$A0=A^*0$.
\end{enumerate}
\end{corollary}

\section*{Acknowledgments}
Heinz Bauschke was partially supported by the Natural Sciences and
Engineering Research Council of Canada and
by the Canada Research Chair Program.
Jonathan  Borwein was partially supported by the Australian Research  Council.
Xianfu Wang was partially supported by the Natural
Sciences and Engineering Research Council of Canada.



\begin{thebibliography}{99}
\bibitem{AtBrezis}
H.\ Attouch and H.\ Br\'{e}zis,
``Duality for the sum of convex functions in general Banach spaces",
\emph{Aspects of Mathematics and its Applications}, J. A.
Barroso, ed., Elsevier Science Publishers, pp.~125--133, 1986.

\bibitem{BB}
H.H.\ Bauschke and J.M.\ Borwein,
``Maximal monotonicity of dense type, local maximal monotonicity,
and monotonicity of the conjugate are all the same for continuous
linear operators'',
\emph{Pacific Journal of Mathematics},
vol.~189, pp.~1--20, 1999.

\bibitem{BBW}
H.H.\ Bauschke, J.M.\ Borwein, and X.\ Wang,
``Fitzpatrick functions and continuous linear
monotone operators'',
\emph{SIAM Journal on Optimization},
vol.~18, pp.~789--809, 2007.



\bibitem{BBWY:1}
H.H.\ Bauschke, J.M.\ Borwein,  X.\ Wang and L.\ Yao,
``For maximally monotone linear relations,
dense type, negative-infimum type, and Fitzpatrick-Phelps type
all coincide with monotonicity of the adjoint'', submitted;
\url{http://arxiv.org/abs/1103.6239v1}, March 2011.

\bibitem{BBWY2}
 H.H.\ Bauschke, J.M.\ Borwein, X.\ Wang, and L.\ Yao,
 ``Every maximally monotone operator of
Fitzpatrick-Phelps type is actually of dense type'',
\emph{Optimization Letters}, in press.

\bibitem{BBWY3}  
H.H.\ Bauschke, J.M.\ Borwein, X.\ Wang, and L.\ Yao,
``Construction of pathological maximally monotone operators on
non-reflexive Banach spaces'', submitted;
\url{http://arxiv.org/abs/1108.1463}, August 2011.


\bibitem{BC2011}
H.H.\ Bauschke and P.L.\ Combettes,
\emph{Convex Analysis and Monotone Operator Theory in Hilbert Spaces},
Springer-Verlag, 2011.




\bibitem{BWY3}
H.H.\ Bauschke, X.\ Wang, and L.\ Yao,
``Monotone linear relations: maximality
and Fitzpatrick functions'',
\emph{Journal of Convex Analysis}, vol.~16, pp.~673--686, 2009.

\bibitem{BWY7}
H.H.\ Bauschke, X.\ Wang, and L.\ Yao,
``Examples of discontinuous
maximal monotone linear operators
and the solution to a recent problem posed by B.F.~Svaiter'',
\emph{Journal of Mathematical Analysis and Applications},
 vol.~370, pp. 224-241,
 2010.

\bibitem{BWY8}
H.H.\ Bauschke, X.\ Wang, and L.\ Yao,
``On Borwein-Wiersma decompositions of monotone linear
relations'', \emph{SIAM Journal on Optimization}, vol.~20, pp.~2636--2652,
 2010.


\bibitem{Borwein}
J.M.\ Borwein, ``Adjoint process duality'',
\emph{Mathematics of Operations Research}, vol.~8, pp.~403--434, 1983.



\bibitem{Bor1}
J.M.\ Borwein,
``Maximal monotonicity via convex analysis'',
\emph{Journal of Convex Analysis}, vol.~13, pp.~561--586, 2006.

\bibitem{Bor2}J.M.\ Borwein, ``Maximality of sums of two maximal monotone operators
 in general
Banach space'',
\emph{Proceedings of the AMS}, vol.~135, pp.~3917--3924, 2007.

\bibitem{Bor3}J.M.\ Borwein, ``Fifty years of maximal monotonicity'',
\emph{Optimization Letters}, vol.~4, pp.~473--490, 2010.







\bibitem{BorVan}
J.M.\ Borwein and J.D.\ Vanderwerff,
\emph{Convex Functions},
Cambridge University Press, 2010.

\bibitem{Brezis70}
H.\ Br\'{e}zis,
``On some degenerate nonlinear parabolic equations'',
in \emph{Nonlinear Functional Analysis
(Proc.\ Sympos.\ Pure Math., vol.~XVIII, part~1,
Chicago, Ill., 1968)},
AMS, pp.~28--38, 1970.

\bibitem{BB76}
H.\ Br\'{e}zis and F.E.\ Browder,
``Singular Hammerstein equations and maximal monotone operators'',
\emph{Bulletin of the AMS},
 vol.~82, pp.~623--625, 1976.


\bibitem{Brezis-Browder}
H.\ Br\'{e}zis and F.E.\ Browder,
``Linear maximal monotone operators
and singular nonlinear integral equations of Hammerstein
type'', in \emph{Nonlinear Analysis (collection of papers in honor of
Erich H.\ Rothe)}, Academic Press, pp.~31--42, 1978.




\bibitem{BurIus}
R.S.\ Burachik and A.N.\ Iusem,
\emph{Set-Valued Mappings and Enlargements of Monotone Operators},
Springer-Verlag, 2008.

\bibitem{Cross}
R.\ Cross, \emph{Multivalued Linear Operators},
Marcel Dekker, Inc, New York, 1998.


\bibitem{FP}
S.\ Fitzpatrick
and R.R.\  Phelps, ``Bounded approximants to monotone operators on Banach spaces'',
 \emph{ Annales de l'Institut Henri Poincar\'{e}.
  Analyse Non Lin\'{e}aire},  vol.~9, pp.~573--595, 1992.



\bibitem{Gossez3}
J.-P.\ Gossez, ``Op\'{e}rateurs monotones non lin\'{e}aires dans
 les espaces de Banach non r\'{e}flexifs'',
\emph{Journal of Mathematical Analysis and Applications}, vol.~34, pp.~371--395,  1971.

\bibitem{Gossez72}
J.-P.\ Gossez,
``On the range of a coercive maximal monotone operator in a nonreflexive
Banach space'',
\emph{Proceedings of the AMS},
vol.~35, pp.~88--92, 1972.

\bibitem{MarSva}
M.\  Marques Alves and B.F.\ Svaiter,
``On Gossez type (D)
maximal monotone operators'',
\emph{Journal of Convex Analysis},
vol.~17, pp.~1077--1088, 2010.


\bibitem{MLegSva}
J.-E.\  Mart\'{i}nez-Legaz and B.F.\ Svaiter,
``Monotone operators representable by l.s.c. convex functions'',
\emph{Set-Valued  Analysis},
vol.~13, pp.~21--46, 2005.

\bibitem{Megg}
R.E.\  Megginson,
\emph{An Introduction to Banach Space Theory},
Springer-Verlag, 1998.


\bibitem{ph}
R.R.\ Phelps,
\emph{Convex Functions, Monotone Operators and
Differentiability},
2nd Edition, Springer-Verlag, 1993.




\bibitem{PheSim}
R.R.\ Phelps and S.\ Simons, ``Unbounded linear monotone
operators on nonreflexive Banach spaces'',
\emph{Journal of Nonlinear and Convex Analysis}, vol.~5, pp.~303--328, 1998.



\bibitem{Rock66}
R.T.\ Rockafellar,
``Extension of Fenchel's duality theorem for
convex functions'',
\emph{Duke Mathematical Journal}, vol.~33, pp.~81--89, 1966.


\bibitem{RockWets}
R.T.\ Rockafellar and R.J-B Wets,
\emph{Variational Analysis}, 3rd Printing,
Springer-Verlag, 2009.



\bibitem{Rudin}
R.\ Rudin,
\emph{Functional Analysis},
Second Edition, McGraw-Hill, 1991.

\bibitem{SiNI}
S.\  Simons,
``The range of a monotone operator'',
\emph{Journal of Mathematical Analysis and Applications}, vol.~199, pp.~176--201,  1996.



\bibitem{Si}
S.\  Simons,
\emph{Minimax and Monotonicity},
Springer-Verlag, 1998.


\bibitem{Si4}S.\  Simons,
``Five kinds of maximal monotonicity'',
\emph{Set-Valued and Variational Analysis},
vol.~9, pp.~391--409, 2001.



\bibitem{Si2}
S.\ Simons, \emph{From Hahn-Banach to Monotonicity},
Springer-Verlag, 2008.


\bibitem{Si3}
S.\ Simons,
``A Br\'{e}zis-Browder theorem for SSDB spaces'', preprint. Available at
\texttt{http://arxiv.org/abs/1004.4251v3}, September 2010.

\bibitem{SiZ}
 S.\ Simons and C.\  Z{\v{a}}linescu, ``Fenchel duality, Fitzpatrick functions
  and maximal monotonicity,''
  \emph{Journal of Nonlinear and Convex Analysis}, vol.~ 6, pp. 1--22, 2005.


\bibitem{Yao}
L.\ Yao,
``The Br\'{e}zis-Browder Theorem revisited and properties of Fitzpatrick functions of order $n$'',
\emph{Fixed Point Theory for Inverse Problems
in Science and Engineering (Banff 2009)}, Springer-Verlag,
  vol.~49, pp.~391--402, 2011.



\bibitem{Zalinescu}
{C.\ Z\u{a}linescu},
\emph{Convex Analysis in General Vector Spaces}, World Scientific
Publishing, 2002.

\bibitem{Zeidlerlin}{E.\ Zeidler},
\emph{Nonlinear Functional  Analysis and its Application,
Vol~II/A Linear Monotone Operators},
Springer-Verlag, New York-Berlin-Heidelberg, 1990.

\bibitem{Zeidler}{E.\ Zeidler},
\emph{Nonlinear Functional  Analysis and its Application,
Vol~II/B Nonlinear Monotone Operators},
Springer-Verlag, New York-Berlin-Heidelberg, 1990.

\end{thebibliography}
\end{document}